\documentclass[graybox]{svmult}

\usepackage{mathptmx}       % selects Times Roman as basic font
\usepackage{helvet}         % selects Helvetica as sans-serif font
\usepackage{courier}        % selects Courier as typewriter font
\usepackage{type1cm}        % activate if the above 3 fonts are
\usepackage{amssymb}
\usepackage{mathrsfs}
\usepackage{amsmath}

\usepackage{makeidx}         % allows index generation
\usepackage{graphicx}        % standard LaTeX graphics tool
\graphicspath{{Figures/}}
\usepackage{multicol}        % used for the two-column index
\usepackage[bottom]{footmisc}% places footnotes at page bottom

\makeindex             % used for the subject index
\begin{document}

\title*{Pseudo-differential operators on $\mathbb{Z}^n$ with applications to discrete fractional integral operators }

\author{Duv\'an Cardona}
\institute{Duv\'an Cardona \at Pontificia Universidad Javeriana, Mathematics Department, Bogot\'a-Colombia. \email{cardonaduvan@javeriana.edu.co}}
\maketitle
\vspace{-2cm}

\abstract*{In this manuscript we provide necessary and sufficient conditions for the $\textnormal{weak}(1,p)$ boundedness, $1< p<\infty,$  of discrete Fourier multipliers (Fourier multipliers  on $\mathbb{Z}^n$). Our main goal is to apply the results obtained to discrete fractional integral operators. Discrete versions of the Calder\'on-Vaillancourt Theorem and the Gohberg Lemma for pseudo-differential operators also are proved. MSC2010: 42B15 (primary), 11P05 (secondary). Keywords: Pseudo-differential operators, Calder\'on-Vaillancourt Theorem, Gohberg Lemma, Discrete fractional integral operator, Hypothesis-$K^*,$ Waring's problem. }

\abstract{In this manuscript we provide necessary and sufficient conditions for the $\textnormal{weak}(1,p)$ boundedness, $1< p<\infty,$  of discrete Fourier multipliers  (Fourier multipliers  on $\mathbb{Z}^n$).  Our main goal is to apply the results obtained to discrete fractional integral operators. Discrete versions of the Calder\'on-Vaillancourt Theorem and the Gohberg Lemma also are proved. MSC2010: 42B15 (primary), 11P05 (secondary).  Keywords: Pseudo-differential operators, Calder\'on-Vaillancourt Theorem, Gohberg Lemma, Discrete fractional integral operator, Hypothesis-$K^*,$ Waring's problem. \\
Submitted exclusively to the {\it{Bulletin of the Iranian Mathematical Society}.}}

\section{Introduction}

\noindent{  {\textbf{Outline of the paper.}}}
In this paper we characterise the $\textnormal{weak}(1,p)$-inequality, $1< p<\infty,$ for multipliers on $\mathbb{Z}^n$ (see \eqref{Zn}). We also prove
discrete versions of the Calder\'on-Vaillancourt Theorem and the Gohberg Lemma
for pseudo-differential operators on $\mathbb{Z}^n.$ Then we apply our results to certain discrete operators called discrete fractional integral operators. Fourier multipliers on $\mathbb{Z}^n$ are defined by the integral representation
\begin{equation}\label{Zn}
  t_m f(n'):=\int\limits_{\mathbb{T}^n} e^{i2\pi n'\cdot\xi}m(\xi)(\mathscr{F}f)(\xi)d\xi,\,\,\,f\in \mathscr{S}(\mathbb{Z}^n),\,n'\in\mathbb{Z}^n,
\end{equation}
where $\mathbb{T}^n=[0,1]^n,$ and $\mathscr{F}:\mathscr{S}(\mathbb{Z}^n)\rightarrow C^{\infty}(\mathbb{T}^n)$ is the discrete Fourier transform, defined on the discrete Schwartz class by
\begin{equation}
  \mathscr{F}f(\xi):=\sum_{n'\in\mathbb{Z}^n}e^{-i2\pi n'\cdot \xi}f(n').
\end{equation}
Same as in the euclidian case, Fourier multipliers on $\mathbb{Z}^n$  are convolution operators and they commute with the group of translations $\{\tau_k\}_{k\in\mathbb{Z}^n},$ where $\tau_kf=f(\cdot-k).$   In \eqref{Zn}, the function $m$ is usually called the symbol of the operator $t_m.$

$\ell^p$-estimates are results of boundedness in $\ell^p$-spaces of suitable linear operators. Although the $L^p$-estimates for multipliers on $\mathbb{R}^n,$ are motivated by its applications to PDE's, (see, e.g., H\"ormander\cite{Hormander1960}) $\ell^p$-bounds for multipliers on $\mathbb{Z}^n$ can be applied to discrete problems, ones arising from difference equations and others from number theory (see \cite{Pie,Ruzhansky}). Taking this into account,  multipliers on $\mathbb{Z}^n$ have been studied principally from three sources. The first one and more classical is as elements in the classic Fourier analysis (see \cite{Duo,GrafakosBook}) as well as discrete counterparts of Calder\'on-Zygmund singular integral operators:
\begin{equation}\label{CZO}
  Tf(x)=\int\limits_{\mathbb{R}^n}k(x-y)f(y)dy,\,\,\,f\in \mathscr{S}(\mathbb{R}^n)
\end{equation}
where the kernel $k$ is singular on the diagonal. In fact, the discrete counterpart of \eqref{CZO} is defined by
\begin{equation}\label{DCZO}
   tf(n')=\sum_{m\in\mathbb{Z}^n,m\neq n'}k(n'-m)f(m),\,\,\,f\in \mathscr{S}(\mathbb{Z}^n).
\end{equation}
In this case, $k(0)=0, and $ the symbol of $t$ is given by $m:=\mathscr{F}k,$ that is the discrete Fourier transform of the kernel. A remarkable result proved in \cite{CalderonZygmund} shows that the $L^p$-boundedness of an operator $T$ implies the $\ell^p$-boundedness of its discrete counterpart $t$. The reference \cite{riesz} includes the same result for the Hilbert transform and the discrete Hilbert transform.

The second viewpoint, in a more recent context, includes discrete multipliers as fundamental examples in the theory of pseudo-differential operators on $\mathbb{Z}^n,$ introduced  in Molahajloo\cite{m}, and developed in the last years by Catana, Delgado, Ghaemi, Jamalpour Birgani, Nabizadeh, Rodriguez and Wong in the references \cite{Cat14,DW,GBN,rod},  as well as in the fundamental work \cite{Ruzhansky} by L. Botchway, G. Kibiti, and M. Ruzhansky where the theory has been motivated by its potential applications to difference equations. In this setting pseudo-differential operators on $\mathbb{Z}^n$ are defined by the integral form
\begin{equation}\label{pseudo}
  t_m f(n'):=\int\limits_{\mathbb{T}^n} e^{i2\pi n'\cdot\xi}m(n',\xi)(\mathscr{F}f)(\xi)d\xi,\,\,\,f\in \mathscr{S}(\mathbb{Z}^n),\,n'\in\mathbb{Z}^n.
\end{equation}
The references Rabinovich\cite{Rab1,Rab2}, and Rabinovich and Roch\cite{Rab3,Rab4} can be think as predecessor works of this subject.  The last approach  is centered in  particular cases of discrete  multipliers and summarized, for example, in the thesis of L.B. Pierce \cite{Pie}. An important type of such operators are called discrete fractional integral operators, and defined by the expression
\begin{equation}\label{FIO}
  I_{k,\lambda+i\gamma}f(n')=\sum_{m=1}^{\infty}\frac{f(n'-m^k)}{m^{\lambda+i\gamma}},\,\,\,f\in \mathscr{S}(\mathbb{Z}),\,n'\in\mathbb{Z}.
\end{equation}
These are discrete analogues of Hardy-Littlewood fractional integral operators of the form
\begin{equation}
  I'_{k,\lambda+i\gamma}f(x)=\int\limits_{1}^{\infty}\frac{f(x-y^k)}{y^{\lambda+i\gamma}}dy.
\end{equation}
The parameter $k$ is a natural number while $0<\lambda\leq 1$ and $\gamma\in\mathbb{R}.$ The attention to $\ell^p$-estimates for the discrete operators $I_{k,\lambda+i\gamma}$ is justified by its nice connections to important problems in number theory as the Waring problem. We refer the reader to \cite{Pie2,st3,Hooley} for a precise discussion in relation with the  Hypothesis-$K^*$ in Waring's problem. This conjecture remains unproved for $k\geq 3$.  Operators as in \eqref{FIO} are multipliers on $\mathbb{Z}^n$ in the sense that the symbol associated to a fixed operator $I_{k,\lambda+i\gamma}$ is given by
\begin{equation}\label{symbol}
  m_{k,\lambda+i\gamma}(\xi)=\sum_{m=1}^{\infty}\frac{e^{-i2\pi m^k\xi}}{m^{\lambda+i\gamma}}.
\end{equation}
Properties of discrete fractional integral operators (and other discrete operators) in $\ell^p$ spaces have been considered in the early works of Stein-Wainger\cite{st,st2,st3}, the thesis of L. B. Pierce\cite{Pie}, Hughes \cite{tn}, and the references \cite{carro,k} and \cite{Pie2}. In particular, a result by Ionescu and Wainger\cite{IW} establish the $\ell^p$-boundedness of $I_{k,\lambda+i\gamma},$ $1<p<\infty,$ for $\lambda=1$ and $\gamma\neq 0.$ The following conjecture is well known in the setting of discrete fractional integral operators (see \cite{Pie2}).\\

\noindent {\bf Conjecture 1}. For $0<\lambda<1$ and $k\in\mathbb{N}$, $I_{k,\lambda}$ extends to a bounded operator from $\ell^q$ into $\ell^p,$ $1\leq q<p<\infty,$  if and only if $p,q$ satisfy
\begin{itemize}
  \item $\frac{1}{p}\leq \frac{1}{q}-\frac{1-\lambda}{k}$,
  \item $\frac{1}{p}<\lambda$ and $\frac{1}{q}>1-\lambda.$
\end{itemize}
We refer the reader to L. B. Pierce \cite{Pie2} where $\ell^q-\ell^p$ estimates of the following type have been proved solving important cases of Conjecture 1:
estimates of the following
type have been proved solving important partial cases of Conjecture 1:
\begin{itemize}
    \item $\frac{1}{p}\leq \frac{1}{q}-\frac{1-\lambda}{\alpha(k)}$, $1-\beta(k)<\lambda<1,$
\end{itemize} where the parameters $\alpha(k)$ and $\beta(k)$ are depending on $p,q$ and $k.$ We refer to \cite[p. 3]{Pie2} for details and to \cite{st3} for sharp expressions of such parameters when $k=2.$ The case $k\geq 4 $ was treated in \cite{Pie} by using the circle method of Hardy-Ramanujan-Littlewood. Let us observe that we have used the order $(\ell^q,\ell^p)$ instead of $(\ell^p,\ell^q)$
used in the previous references.

\noindent{  {\textbf{Main results.}}} In this paper, we  determinate those conditions on the symbol $m$ in order that the operator $t_m$ can be extended to a bounded operator from $\ell^1(\mathbb{Z}^n)$ into $\ell^{p,\infty}(\mathbb{Z}^n),$ $1< p<\infty;$ this means that the following inequality
\begin{equation}\label{weak1p}
 \Vert t_mf \Vert_{\ell^{p,\infty}}:=\sup_{\alpha>0}\alpha\cdot \mu\{n:|t_mf(n)|>\alpha\}^{\frac{1}{p}}\leq C\Vert f\Vert_{\ell^1}
\end{equation}
remains valid for some constant $C>0$ and all $f\in\ell^1.$ We have denoted by $\mu$ the counting measure on the discrete space $\mathbb{Z}^n.$ Then we will apply the obtained results to the operators $I_{k,\lambda+i\gamma}.$ In Theorem \ref{t1dc}, we present the following characterisation for $t_m$ in terms of the inverse Fourier transform $k:=\mathscr{F}^{-1}m$ of $m,$
\begin{itemize}
    \item $t_m:\ell^1(\mathbb{Z}^n)\rightarrow\ell^{p,\infty}(\mathbb{Z}^n)$ extends to a bounded operator if and only if $k:=\mathscr{F}^{-1}m\in \ell^{p,\infty}(\mathbb{Z}^n).$ Moreover, $
    \Vert k\Vert_{\ell^{p,\infty}}= \Vert t_m \Vert_{\mathscr{B}(\ell^1,\ell^{p,\infty})}$.
  \item $t_m:\ell^1(\mathbb{Z}^n)\rightarrow\ell^{p}(\mathbb{Z}^n)$ extends to a bounded operator if and only if $k:=\mathscr{F}^{-1}m\in \ell^{p}(\mathbb{Z}^n).$ Moreover, $
    \Vert k\Vert_{\ell^{p}}= \Vert t_m \Vert_{\mathscr{B}(\ell^{1},\ell^p)}$.
  \end{itemize}
This result is a discrete version of the characterisation for the weak $(1,p)$ inequality for multipliers on $\mathbb{R}^n$ due to Stepanov-Haydy-Littlewood and Sobolev(see \cite{Stepanov}). As a consequence of the previous result we obtain our main result for discrete fractional integral operators.
\begin{theorem}\label{T2}
 Let us consider $1< p<\infty,$ $k\in\mathbb{N},$ $0<\lambda<1,$  and $\gamma\in \mathbb{R}.$ Then,
 \begin{itemize}
 \item $I_{k,\lambda+i\gamma}$ is of weak type $(1,p)$,
  if and only if $\lambda\geq \frac{1}{p}.$
 \item $I_{k,\lambda+i\gamma}$ is bounded from  $\ell^1$ into $\ell^p$
  if and only if $\lambda> \frac{1}{p}.$
 \end{itemize}

\end{theorem}
Let us observe that the second assertion of Theorem \ref{T2} show that the  Conjecture 1 holds true in the partial case $q=1$ and $1<p<\infty$ because $\frac{1}{p}<\lambda\leq1-\frac{1-\lambda}{k}, $ for all $k\geq 1.$

To determinate sharp conditions for the $\ell^2$-boundedness of pseudo-differential operators on $\mathbb{Z}^n$ we will prove a discrete version of a classical result due to A. Calder\'on and R. Vaillancourt, to do so our main tool is the following result:
\begin{itemize}
  \item the $L^2(\mathbb{R}^n)$-boundedness of a pseudo-differential operator on $\mathbb{R}^n,$ defined by
\begin{equation}\label{pseudorn}
  Tf(x)=\int_{\mathbb{R}^n}e^{i2\pi x\xi}a(x,\xi)\widehat{f}(\xi)\,d\xi,\,\,\,f\in\mathscr{S}(\mathbb{R}^n),
\end{equation} (here $\widehat{f}(\xi)=(\mathscr{F}_{\mathbb{R}^n}f)(\xi)=\int_{\mathbb{R}^n}e^{i\pi x\xi}f(x)dx,$ denotes the  Fourier transform of $f$) implies the $\ell^2(\mathbb{Z}^n)$-boundedness of its discrete analogue
 \begin{equation}\label{pseudornzn}
  t_mf(n')=\int_{\mathbb{T}^n}e^{i2\pi x\xi}m(n',\xi)(\mathscr{F}{f})(\xi)\,d\xi,\,\,\,f\in\mathscr{S}(\mathbb{Z}^n),\,\,m(n',\xi)=\overline{a(\xi,-n')},
\end{equation} where the symbol $a$ is assumed to be  a continuous function.
\end{itemize}
  Consequently, we prove the following theorem.
  \begin{theorem}[Calder\'on-Vaillancourt, discrete version]
    Let us assume that $t_m$ is a pseudo-differential operator on $\mathbb{Z}^n.$ Then, under the condition \begin{equation}\label{calderon}
  |\partial_x^\beta\Delta_\xi^\alpha m(x,\xi)|\leq C_{\alpha,\beta}(1+|\xi|)^{(|\beta|-|\alpha|)\rho},\,\,0\leq \rho<1\,\,
\end{equation} the operator $ t_m $ extends to a bounded operator on $\ell^2(\mathbb{Z}^n).$
  \end{theorem}

An analogue result to the previous theorem has been proved in the work \cite{Ruzhansky} by L. Botchway, G. Kibiti, and M. Ruzhansky, where conditions of the type $|\partial_x^\beta m(x,\xi)|\leq C_\beta,$ $|\beta|\leq [\frac{n}{2}]+1,$ are imposed on the symbols in order to obtain the $\ell^2$-boundedness of pseudo-differential operators. As an immediate consequence of the techniques developed in \cite{Ruzhansky} we present the following characterisation for the $\ell^2$-compactness of discrete operators:
\begin{itemize}
  \item (Gohberg Lemma, discrete version) a pseudo-differential operator $t_m$ extends to a compact operator on $\ell^2(\mathbb{Z}^n)$ if and only if
      \begin{eqnarray}
      \limsup_{|n'|\rightarrow\infty}\sup_{\xi\in\mathbb{T}^n}|m(n',\xi)|=0.
      \end{eqnarray}
\end{itemize} 
This paper is organized as follows. In the next section we include some preliminaries and later  we prove our main results concerning to multipliers and discrete fractional operators. We end Section \ref{sectionweak} with Remark \ref{Rem} where we discuss our results in relation with the Hypothesis $K^*$ conjectured by Hooley on the Waring problem. Finally, in Section \ref{sectionpseudo} we prove our discrete versions of the Calder\'on-Vaillancourt theorem and the Gohberg Lemma.

\section{ $\textnormal{Weak-}\ell^p$ estimates for multipliers }\label{sectionweak}

In this section we prove our main theorems. However, we recall some basics on Lebesgue spaces and weak Lebesgue spaces on $\mathbb{Z}^n.$
For every $n\in\mathbb{N},$ we denote by $\ell^p(\mathbb{Z}^n)$ to the set of complex functions on $\mathbb{Z}^n$ satisfying
\begin{equation}
  \Vert f\Vert_{\ell^p}:=(\sum_{n'\in\mathbb{Z}^n}|f(n')|^p)^{\frac{1}{p}}<\infty,
\end{equation}
if $1\leq p<\infty.$ The $\textnormal{weak-}\ell^p$ space on $\mathbb{Z}^n,$ $\ell^{p,\infty}(\mathbb{Z}^n),$ $1\leq p<\infty,$ is defined by those functions $f$ on $\mathbb{Z}^n$ satisfying
\begin{equation}
 \Vert f\Vert_{\ell^{p,\infty}} :=\sup_{\alpha>0}\alpha \mu\{n':|f(n')|>\alpha  \}^{\frac{1}{p}}<\infty.
\end{equation} If we fix $0<r<p,$ we have the following seminorm on $\ell^{p,\infty},$ (see e.g., L. Grafakos \cite{GrafakosBook}, p. 13):
\begin{equation}\label{semi}
  \Vert f \Vert'_{\ell^{p,\infty}}:=\sup_{E\subset \mathbb{Z}^n,\,1\leq \mu(E)<\infty}\mu(E)^{\frac{1}{p}-\frac{1}{r}}(\sum_{n'\in E}|f(n')|^r)^{\frac{1}{r}}
\end{equation}
which satisfies
\begin{equation}
  \Vert f\Vert_{\ell^{p,\infty}}\leq \Vert f \Vert'_{\ell^{p,\infty}}\leq (\frac{p}{p-r})^{\frac{1}{r}}\Vert f\Vert_{\ell^{p,\infty}}.
\end{equation}
Now, we start with the following lemma.
\begin{lemma}\label{Lemma1}
  Let us consider the multiplier $t_m$ with symbol associated $m$ defined on $\mathbb{Z}^n.$ If $t_m:\ell^1\rightarrow \ell^{p,\infty}$ is bounded, then $k\in\ell^{p,\infty}$ and \begin{equation}\label{T11}
         \Vert k\Vert_{\ell^{p,\infty}}  \leq  \Vert t_m\Vert_{\mathscr{B}(\ell^1,\ell^{p,\infty})}<\infty.
       \end{equation}
\end{lemma}
\begin{proof}
  If $t_m:\ell^1\rightarrow \ell^{p,\infty}$ is bounded, we write
  \begin{equation}
    \mu\{s: |t_mf(s)|>\alpha  \}\leq(C\Vert f \Vert_{\ell^1}/\alpha)^p,\,\,C:=\Vert t_m\Vert_{\mathscr{B}(\ell^1,\ell^{p,\infty})}.
  \end{equation} In particular, if  $f=1_{\{0\}}$ is the characteristic function of the set $\{0\},$ then $t_mf=k\ast f=k.$ So,   we obtain
  \begin{align*}
    &\mu\{s: |t_mf(s)|>\alpha  \}=\mu\{s: |k(s)|>\alpha  \}\leq(C/\alpha)^p,\,\,C:=\Vert t_m\Vert_{\mathscr{B}(\ell^1,\ell^{p,\infty})}.
  \end{align*}
  Thus, we obtain
  \begin{equation}
         \Vert k\Vert_{\ell^{p,\infty}}  \leq  \Vert t_m\Vert_{\mathscr{B}(\ell^1,\ell^{p,\infty})}<\infty.
       \end{equation}
   So, we end the proof.
\end{proof}

Now, we prove  the following theorem.
\begin{theorem}\label{t1dc}
   Let $1< p<\infty$ and let us consider a multiplier $t_m$ with associated symbol $m.$ Let us consider the convolution kernel of $t_m$ given by the inverse Fourier transform $k:=\mathscr{F}^{-1}m$ of m.
  Then
  \begin{itemize}
    \item $t_m:\ell^1(\mathbb{Z}^n)\rightarrow\ell^{p,\infty}(\mathbb{Z}^n)$ extends to a bounded operator if and only if $k:=\mathscr{F}^{-1}m\in \ell^{p,\infty}(\mathbb{Z}^n).$ Moreover, $
    \Vert k\Vert_{\ell^{p,\infty}}= \Vert t_m \Vert_{\mathscr{B}(\ell^1,\ell^{p,\infty})}$.
  \item $t_m:\ell^1(\mathbb{Z}^n)\rightarrow\ell^{p}(\mathbb{Z}^n)$ extends to a bounded operator if and only if $k:=\mathscr{F}^{-1}m\in \ell^{p}(\mathbb{Z}^n).$ Moreover, $
    \Vert k\Vert_{\ell^{p}}= \Vert t_m \Vert_{\mathscr{B}(\ell^{1},\ell^p)}$.
  \end{itemize}
\end{theorem}
\begin{proof}
  Let us consider $1< p<\infty.$ In view of Lemma \eqref{Lemma1}, if $t_m:\ell^1\rightarrow \ell^{p,\infty}$ is bounded then \begin{equation}\label{step0}
         \Vert k \Vert_{\ell^{p,\infty}}=\Vert\mathscr{F}^{-1}(m)(\cdot) \Vert_{\ell^{p,\infty}}\leq \Vert t_m\Vert_{\mathscr{B}(\ell^1,\ell^{p,\infty})}.
       \end{equation}
  So, for the  proof, we only need to show that the condition  $k:=\mathscr{F}^{-1}m\in \ell^{p,\infty}(\mathbb{Z})$ implies the boundedness of $t_m $ from $\ell^1(\mathbb{Z}^n)$ into $\ell^{p,\infty}(\mathbb{Z}^n)$. So, let us assume $k:=\mathscr{F}^{-1}m\in \ell^{p,\infty}(\mathbb{Z}^n).$ We will show the inequality
  \begin{equation}\label{step1}
    \Vert t_mf(n)\Vert_{{\ell}^{p,\infty}} \leq C\Vert f\Vert_{\ell^1}, C=\Vert k \Vert_{\ell^{p,\infty}},
  \end{equation}
  for every $f\in\ell^1.$   But this is consequence of the weak Young inequality, (also called Hardy-Littlewood-Sobolev inequality)
  \begin{equation}\label{final}
    \Vert t_m f\Vert_{\ell^{p,\infty}}=\Vert k\ast f\Vert_{\ell^{p,\infty}}\leq \Vert k \Vert_{\ell^{p,\infty}}\Vert f\Vert_{\ell^1}
  \end{equation} because it implies
  \begin{equation}
    \Vert t_m\Vert_{\mathscr{B}(\ell^1,\ell^{p,\infty})}\leq \Vert k\Vert_{\ell^{p,\infty}}.
  \end{equation} So, we end the proof for the first item. For the second part, if $t_m:\ell^1\rightarrow\ell^p$ is bounded, then
  \begin{equation}\label{l1lp}
    \Vert t_m f\Vert_{\ell^p}\leq \Vert t_m\Vert_{\mathscr{B}(\ell^1,\ell^{p})}\Vert f\Vert_{\ell^1}.
  \end{equation} In particular, if $f=1_{ \{0\} },$ $t_mf=k $ and we obtain $\Vert k\Vert_{\ell^p}\leq \Vert t_m\Vert_{\mathscr{B}(\ell^1,\ell^{p})}.$ For the converse assertion, we only need to apply the Young inequality, in fact the estimate
\begin{equation}\label{final2}
    \Vert t_m f\Vert_{\ell^{p}}=\Vert k\ast f\Vert_{\ell^{p}}\leq \Vert k \Vert_{\ell^{p}}\Vert f\Vert_{\ell^1}
  \end{equation} implies  $\Vert k\Vert_{\ell^p}\geq \Vert t_m\Vert_{\mathscr{B}(\ell^1,\ell^{p})}.$ So, we finish the proof of the theorem.
\end{proof}

Now, we apply the above result  to the analysis of discrete fractional integral operators.

\begin{theorem}
 Let us consider $1< p<\infty,$ $k\in\mathbb{N},$ $0<\lambda\leq 1,$  and $i\gamma\in i\mathbb{R}.$ Then,
 \begin{itemize}
   \item  $I_{k,\lambda+i\gamma}$ is of weak type $(1,p)$,
  if and only if $\lambda\geq \frac{1}{p}.$
   \item  $I_{k,\lambda+i\gamma}$ is bounded from $\ell^1$ into  $\ell^p$,
  if and only if $\lambda> \frac{1}{p}.$
 \end{itemize}

\end{theorem}
\begin{proof}
  We want to proof the assertion by using Theorem \ref{t1dc}. The convolution kernel of $I_{k,\lambda+i\gamma},$ which we denote by $k_{k,\lambda+i\gamma},$ is defined on $\mathbb{Z}$ by: $k_{k,\lambda+i\gamma}(m^k)=\frac{1}{m^{\lambda+i\gamma}}$ for $m>0,$ and  $k_{k,\lambda+i\gamma}(s)=0$ if $s\neq m^k$ for every $m>0.$ Let us recall that $k_{k,\lambda+i\gamma}\in \ell^{p,\infty}$ if and only if
  \begin{align*}
  \Vert k_{k,\lambda+i\gamma} \Vert_{\ell^{p,\infty}}:=\sup_{\alpha>0}\alpha\mu\{s\in\mathbb{Z}:|k_{k,\lambda+i\gamma}(s)|>\alpha\}^{\frac{1}{p}}<\infty.
  \end{align*}
  Taking into account that
   \begin{align*}
  \Vert k_{k,\lambda+i\gamma} \Vert_{\ell^{p,\infty}}&:=\sup_{\alpha>0}\alpha\mu\{s\in\mathbb{Z}:|k_{k,\lambda+i\gamma}(s)|>\alpha\}^{\frac{1}{p}}\\
  &=\sup_{\alpha>0}\alpha\{m^{k}\in \mathbb{Z}: m>0\,\,\textnormal{and}\,\,\frac{1}{m^\lambda}>\alpha \}^{\frac{1}{p}}\\
  &= \sup_{\alpha>0}\alpha\mu\{m^{k}\in \mathbb{Z}: m>0\,\,\textnormal{and}\,\,m <\frac{1}{\alpha^{\frac{1}{\lambda}}} \}^{\frac{1}{p}}\\
  \end{align*}  and by considering that
  $ \mu\{m^{k}\in \mathbb{Z}: m>0\,\,\textnormal{and}\,\,m <\frac{1}{\alpha^{\frac{1}{\lambda}}} \}=0, $ for  $\alpha>1,$
we deduce the following fact
  \begin{align*}
    \sup_{\alpha>0}\alpha\mu\{s\in\mathbb{Z}:|k_{k,\lambda+i\gamma}(s)|>\alpha\}^{\frac{1}{p}}
    &= \sup_{0<\alpha\leq 1} \mu\{m^{k}\in \mathbb{Z}: m>0\,\,\textnormal{and}\,\,m <\frac{1}{\alpha^{\frac{1}{\lambda}}} \} \\ &=\sup_{0<\alpha\leq 1}\alpha[\frac{1}{\alpha^{\frac{1}{\lambda}}}]^{\frac{1}{p}}
  \end{align*}
  where $[\cdot]$ denotes the integer part function.
  Now, because
  $$\sup_{0<\alpha\leq 1}\alpha\left[\frac{1}{\alpha^{\frac{1}{\lambda}}}\right]^{\frac{1}{p}}=\sup_{0<\alpha\leq 1}\alpha^{1-\frac{1}{\lambda p}}\leq C<\infty,$$
if and only if $\lambda p\geq 1$ we finish the proof of the first assertion. Now, for the second part, let us note that $k_{k,\lambda+i\gamma}\in \ell^p$ if and only if $\lambda p>1.$ So, by Theorem \eqref{t1dc} we end the proof.
\end{proof}

We end this section with the following discussion in relation with  the Hypothesis $K^*.$
\begin{remark}\label{Rem}
It was proved in Stein and Wainger\cite{st3}, that the Hypothesis $K^*$ conjectured by Hooley, is equivalent to the following fact:
\begin{equation}\label{K*}
  \textnormal{for all\,} \lambda: \frac{1}{2}<\lambda<1,\textnormal{for all\,} k\in\mathbb{N}:\,\, m_{k,\lambda}\in L^{2k}[0,1].
\end{equation} By using Theorem \ref{T2} we can show the following (weak) assertion:
\begin{equation}\label{K}
  \textnormal{for all\,} \lambda:\frac{1}{2}<\lambda<1,\textnormal{for all\,} k\in\mathbb{N}:\,\,k<\frac{1/2}{1-\lambda}, \,\, m_{k,\lambda}\in L^{2k}[0,1].
\end{equation} In fact, by Theorem \ref{T2}, for every $\frac{1}{2}<\lambda<1$ the operator $I_{k,\lambda}$ is  weak$(1,r')$  where $r':=\frac{1}{\lambda}.$ By Theorem \ref{T2} the convolution kernel of $I_{k,\lambda},$ $k_{k,\lambda}=\mathscr{F}^{-1}m_{k,\lambda}\in \ell^{r',\infty}.$  Because $1<r'<2,$ by the Hardy-Littlewood-Stein inequality, we obtain:
\begin{equation}\label{LPE}
  \Vert m_{k,\lambda}\Vert_{L^{r,\infty}[0,1]}\leq C\Vert k_{k,\lambda}\Vert_{r',\infty}.
\end{equation} So, $m_{k,\lambda}\in L^{r,\infty}[0,1]$ and by the embedding $L^{r,\infty}[0,1]\subset L^{s}[0,1]$ for $s<r=\frac{1}{1-\lambda}$ we deduce that $m_{k,\lambda}\in L^{2k}[0,1]$ for $k< \frac{1/2}{1-\lambda}.$
\end{remark}

\section{Pseudo-differential operators on $\ell^2(\mathbb{Z}^n).$ Calder\'on-Vaillancourt Theorem and Gohberg Lemma }\label{sectionpseudo}

In this section we investigate the action of discrete pseudo-differential operators on $\ell^2(\mathbb{Z}^n).$ Our starting point is the following result where we show that the boundedness of pseudo-differential operators can be transferred to their discrete analogues.

\begin{theorem}\label{teorema principal1}
Let $1< p<\infty.$ If  $a:\mathbb{T}^n\times\mathbb{R}^n\rightarrow \mathbb{C}$ is a continuous bounded function and the pseudo-differential operator
\begin{equation}
Tf(x)=\int_{\mathbb{R}^n}e^{i2\pi(x,\xi)}a(x,\xi)(\mathscr{F}_{\mathbb{R}^n}f)(\xi)d\xi
\end{equation} extends to a bounded operator on $L^2(\mathbb{R}^n),$ then the discrete pseudo-differential operator
\begin{equation}\label{cp}
t_mf(n'):=\int_{\mathbb{T}^n}e^{i2\pi(x,\xi)}a(x,\xi)(\mathscr{F}f)(\xi)d\xi,\,\,m(n',\xi)=\overline{a(\xi,-n')},
\end{equation} also extends to a bounded operator on $\ell^2(\mathbb{Z}^n).$ Moreover, some constant $C_p$ satisfies $\Vert t_m \Vert_{\mathscr{B}(\ell^p(\mathbb{Z}^n))}\leq C_p\Vert T\Vert_{\mathscr{B}(L^p(\mathbb{R}^n))}.$
\end{theorem}
We begin with the proof of this result by considering the following technical  lemma.
\begin{lemma}\label{lemma1} Suppose $f$ is a continuous periodic function on $\mathbb{R}^n$ and let $\{g_{m}\}$ be a sequence of uniformly bounded continuous periodic functions on $\mathbb{R}^n.$ If $g_{m}$ converges pointwise to a function $g$ defined on $\mathbb{R}^n$ and $\epsilon_m$ is a positive sequence of real numbers, then
\begin{equation}
\lim_{m\rightarrow\infty}\epsilon_m^{\frac{n}{2}}\int_{\mathbb{R}^n}e^{-\epsilon_m|x|^2}f(x)g_{m}(x)dx=\int_{\mathbb{T}^n}f(x)g(x)dx
\end{equation} provided that $\epsilon_{m}\rightarrow 0.$
\end{lemma}
\begin{proof}
By using Lemma 3.9 of \cite{SteWei} we have for every $m\in\mathbb{N}$
$$I_{1,m}:=\lim_{s\rightarrow\infty}  \epsilon_s^{\frac{n}{2}}\int_{\mathbb{R}^n}e^{-\epsilon_s|x|^2}f(x)g_{m}(x)dx=\int_{\mathbb{T}^n}f(x)g_m(x)dx.                  $$ Now, taking into account that the sequence $\{g_m\}$ is uniformly bounded, an application of the dominated convergence  theorem gives
$$I_{2,m}:=\lim_{m\rightarrow\infty}  \epsilon_s^{\frac{n}{2}}\int_{\mathbb{R}^n}e^{-\epsilon_s|x|^2}f(x)g_{m}(x)dx=\epsilon_s^{\frac{n}{2}}\int_{\mathbb{T}^n}e^{-\epsilon_s|x|^2}f(x)g(x)dx.                  $$
So, the limit $\lim_{m,s\rightarrow\infty}  \epsilon_s^{\frac{n}{2}}\int_{\mathbb{R}^n}e^{-\epsilon_s|x|^2}f(x)g_{m}(x)dx                 $ there exists and can be computed from iterated limits in the following way
$$ \lim_{m,s\rightarrow\infty}  \epsilon_s^{\frac{n}{2}}\int_{\mathbb{R}^n}e^{-\epsilon_s|x|^2}f(x)g_{m}(x)dx =\lim_{m\rightarrow\infty} I_{1,m}=\int_{\mathbb{T}^n}f(x)g(x)dx,$$
where in the last line we have use the dominated convergence theorem.
Consequently we obtain
$$\lim_{m\rightarrow\infty}  \epsilon_m^{\frac{n}{2}}\int_{\mathbb{R}^n}e^{-\varepsilon_m|x|^2}f(x)g_{m}(x)dx=\int_{\mathbb{T}^n}f(x)g(x)dx.                  $$ So, we finish the proof.
\end{proof}

\begin{proof} Our proof consists of several steps. First, we will show that the $L^2(\mathbb{R}^n)$-boundedness of $T$ implies the $L^2(\mathbb{T}^n)$-boundedness of the periodic pseudo-differential operator
\begin{equation}\label{Periodic}
  Af(x)=\sum_{\xi\in\mathbb{T}^n}e^{i2\pi x\cdot \xi}a(x,\xi)(\mathscr{F}_{\mathbb{T}^n}f)(\xi),\,\,f\in \mathscr{D}(\mathbb{T}^n).
\end{equation} where $\mathscr{F}_{\mathbb{T}^n}f(\xi)=\int_{\mathbb{T}^n}e^{i2\pi(x,\xi)}{f(x)}dx,$ is the periodic Fourier transform of $f.$  Later we will prove that the $L^2(\mathbb{T}^n)$-boundedness of $A$ implies the $\ell^2(\mathbb{Z}^n)$-boundedness of $t_m.$\\

Now,  let us assume that $P$ and $Q$ are trigonometric polynomials. For every $\delta>0$ let us denote by $w_\delta(x)=e^{-\delta|x|^2}.$ So, if $\varepsilon,\alpha,\beta>0$ and $\alpha+\beta=1$ let us note that
\begin{equation}\label{eq1}
\lim_{\varepsilon \rightarrow 0}\varepsilon^{\frac{n}{2}}\int_{\mathbb{R}^n}T(Pw_{\varepsilon\alpha})\overline{Q(x)}w_{\varepsilon \beta}dx=(\pi/\beta)^{n/2}\int_{\mathbb{T}^n}(A P)(x)\overline{Q(x)}dx.
\end{equation}
By linearity we only need to prove \eqref{eq1} when $P(x)=e^{i2\pi m x}$ and $Q(x)=e^{i2\pi k x}$ for $k$ and $m$ in $\mathbb{Z}^n.$ The right hand side of \eqref{eq1} can be computed as follows,
\begin{align*}
\int_{\mathbb{T}^n}(A P)(x)\overline{Q(x)}dx &=\int_{\mathbb{T}^n}\left(\sum_\xi e^{i2\pi(x,\xi)}{a(x,\xi)\delta_{m,\xi}} \right)\overline{Q(x)}dx\\
&=\int_{\mathbb{T}^n} e^{i2\pi(x,m)}a(x,m)\overline{Q(x)}dx=\int_{\mathbb{T}^n} e^{i2\pi(x,m)-i2\pi kx }a(x,m)dx.
\end{align*}
Now, we compute the left hand side of \eqref{eq1}. Taking under consideration that the euclidean Fourier transform of $P(x)w_{\alpha\varepsilon}$ is given by
\begin{equation}
\mathscr{F}_{\mathbb{R}^n}(Pw_{\alpha\varepsilon})(\xi)=(\alpha\varepsilon)^{-\frac{n}{2}}e^{-|\xi-m|^2/\alpha\varepsilon},
\end{equation}
by the Fubini theorem we have
\begin{align*}
\int_{\mathbb{R}^n}    T(Pw_{\varepsilon\alpha})   & \overline{Q(x)}  w_{\varepsilon \beta}(x)dx =\int_{\mathbb{R}^n}\int_{\mathbb{R}^n}e^{i2\pi( x,\xi)}a(x,\xi) (\alpha\varepsilon)^{-\frac{n}{2}}e^{-|\xi-m|^2/\alpha\varepsilon} \overline{Q(x)}w_{\varepsilon \beta}(x)d\xi dx\\
&=\int_{\mathbb{R}^n}\int_{\mathbb{R}^n}e^{i2\pi( x,\xi)-i2\pi kx} e^{-\pi\varepsilon\beta|x|^2} a(x,\xi) dx (\alpha\varepsilon)^{-\frac{n}{2}}e^{-|\xi-m|^2/\alpha\varepsilon} d\xi\\
&=\int_{\mathbb{R}^n}\int_{\mathbb{R}^n}e^{i2\pi( x,(\alpha\varepsilon)^{\frac{1}{2}}\eta+m)-i2\pi kx} a( x,(\alpha\varepsilon)^{\frac{1}{2}} \eta+m) e^{-\pi\varepsilon\beta|x|^2}dx  e^{-|\eta|^2} d\eta.\\
\end{align*}
So, we have
\begin{align*}
&\lim_{\varepsilon\rightarrow 0}\varepsilon^{n/2} \int_{\mathbb{R}^n}    T(Pw_{\varepsilon\alpha})   \overline{Q(x)}  w_{\varepsilon \beta}(x)dx =\lim_{\varepsilon\rightarrow 0}\beta^{-\frac{n}{2}}(\beta\varepsilon)^{n/2}\int_{\mathbb{R}^n}    T(Pw_{\varepsilon\alpha})    \overline{Q(x)}  w_{\varepsilon \beta}(x)dx\\
&=\lim_{\varepsilon\rightarrow 0}\beta^{-\frac{n}{2}}(\beta\varepsilon)^{n/2} \int_{\mathbb{R}^n}\int_{\mathbb{R}^n}e^{i2\pi( x,(\alpha\varepsilon)^{\frac{1}{2}}\eta+m)-i2\pi kx}a( x,(\alpha\varepsilon)^{\frac{1}{2}} \eta+m) e^{-\pi\varepsilon\beta|x|^2}dx \,\cdot\, e^{-|\eta|^2} d\eta.
\end{align*}
By Lemma \eqref{lemma1}, we have
\begin{align*}
\lim_{\varepsilon\rightarrow 0} (\beta\varepsilon)^{n/2} &\int_{\mathbb{R}^n}e^{i2\pi( x,(\alpha/\beta)^{\frac{1}{2}}(\beta\varepsilon)^{\frac{1}{2}}\eta+m)}e^{-i2\pi kx}a( x,(\alpha\varepsilon)^{\frac{1}{2}} \eta+m) e^{-\pi\varepsilon\beta|x|^2}dx\\
&=\int_{\mathbb{T}^n}e^{i2\pi( x,m)-i2\pi kx}a(x,m)dx.
\end{align*}
Taking into account that $\int_{\mathbb{R}^n}e^{-|\eta|^2}d\eta=\pi^{n/2},$ and that $a$ is a  continuous bounded function, by the dominated convergence theorem we have
\begin{align*}
\lim_{\varepsilon\rightarrow 0}\varepsilon^{n/2} \int_{\mathbb{R}^n}    T(Pw_{\varepsilon\alpha})   \overline{Q(x)}  w_{\varepsilon \beta}(x)dx =(\pi/\beta)^{n/2}\int_{\mathbb{T}^n}e^{i2\pi( x,m)-i2\pi kx}a(x,m)dx.
\end{align*}
If we assume that $T$ is a bounded linear operator on $L^2(\mathbb{R}^n),$ then the restriction of $A$ to trigonometric polynomials is a bounded operators on $L^2(\mathbb{T}^n).$ In fact, if $\alpha=\frac{1}{2}$ and $\beta=\frac{1}{2}$ we obtain
\begin{align*}
&\Vert AP\Vert_{L^2(\mathbb{T}^n)} =\sup_{\Vert Q \Vert_{L^{2}(\mathbb{T}^n)}=1}\left|\int_{\mathbb{T}^n}(A P)(x)\overline{Q(x)}dx \right|\\
&=\sup_{\Vert Q \Vert_{L^{2}(\mathbb{T}^n)}=1}\lim_{\varepsilon\rightarrow 0} \varepsilon^{n/2} (\frac{1}{\pi 2})^{n/2}\left| \int_{\mathbb{R}^n}    T(Pw_{\varepsilon\alpha})   \overline{Q(x)}  w_{\varepsilon \beta}(x)dx  \right|\\
&\leq\sup_{\Vert Q \Vert_{L^{2}(\mathbb{T}^n)}=1}\lim_{\varepsilon\rightarrow 0} \varepsilon^{n/2}  (\frac{1}{\pi 2})^{n/2}\Vert T \Vert_{\mathscr{B}(L^2)}\Vert Pw_{\varepsilon/2}\Vert_{L^2(\mathbb{R}^n)}\Vert Qw_{\varepsilon/2}\Vert_{L^{2}(\mathbb{T}^n)}\\
&\leq\sup_{\Vert Q \Vert_{L^{2}(\mathbb{T}^n)}=1} \Vert T \Vert_{\mathscr{B}(L^2)}\lim_{\varepsilon\rightarrow 0}  (\frac{1}{\pi 2})^{n/2} \left(    \varepsilon^{n/2}\int_{\mathbb{R}^n}|P(x)|^2e^{-\pi\varepsilon|x|^2}dx\right)^{\frac{1}{2}} \\
&\hspace{8cm} \times\left(  \varepsilon^{n/2}  \int_{\mathbb{R}^n}|Q(x)|^{2}e^{-\pi\varepsilon|x|^2}dx   \right)^{\frac{1}{2}}\\
&\leq\sup_{\Vert Q \Vert_{L^{2}(\mathbb{T}^n)}=1} \Vert T \Vert_{\mathscr{B}(L^2)} (\frac{1}{\pi 2})^{n/2} \left(  \int_{\mathbb{T}^n}|P(x)|^2dx\right)^{\frac{1}{2}}  \left(   \int_{\mathbb{T}^n}|Q(x)|^{2}dx   \right)^{\frac{1}{2}}\\
&= \Vert T \Vert_{\mathscr{B}(L^2)} (\frac{1}{\pi 2})^{n/2} \Vert P\Vert_{L^2(\mathbb{T}^n)}.
\end{align*}
Because the restriction of $A$ to trigonometric polynomials is a bounded operator on $L^2(\mathbb{T}^n)$ this restriction admits a unique bounded extension on $L^2(\mathbb{T}^n).$ Now, by equation (4.4) of Theorem 4.1 of \cite{Ruzhansky} we have the identity
\begin{equation}\label{Ruzidentity}
  t_m=\mathscr{F}^{-1}A^{*}\mathscr{F},\,\,\,m(n',\xi)=\overline{a(\xi,-n')}.
\end{equation}
Because, $\mathscr{F}:\ell^2(\mathbb{Z}^n)\rightarrow L^2(\mathbb{T}^n)$ extends to an isomorphism of Hilbert spaces, $t_m$ extends to a bounded operator on $\ell^2$ if and only if $A^*$ (and hence $A$) is $L^2$-bounded. So, we finish the proof.
\end{proof}

Now, by using the previous result and the following theorem we establish our discrete version of the Calder\'on-Vaillancourt theorem (see \cite{ca1,ca2}).

\begin{theorem}[Calder\'on-Vaillancourt]\label{CalVaillTheo}
    Let us assume that $T$ is a pseudo-differential operator on $\mathbb{R}^n$ associated with the symbol $a.$ Then, under the condition \begin{equation}\label{calderon}
  |\partial_x^\beta\partial_\xi^\alpha a(x,\xi)|\leq C_{\alpha,\beta}(1+|\xi|)^{(|\beta|-|\alpha|)\rho},\,\,0\leq \rho<1\,\,
\end{equation} the operator $ T $ extends to a bounded operator on $L^2(\mathbb{R}^n).$
  \end{theorem}
 The Calder\'on-Vaillancourt Theorem is sharp. In fact, this theorem fails when $\rho=1$ (see Duoandikoetxea \cite{Duo}, pag. 113). Because pseudo-differential operators are associated to symbol with a discrete first argument we need the notion of discrete derivatives. So,  we define the  partial difference operators $\Delta_{\xi_j}$ by:
$
\Delta_{\xi_j}\sigma(\xi)= \sigma(\xi+\delta_j)-\sigma(\xi),
$ and $
\Delta_{\xi}^{\alpha}=\Delta_{\xi_1}^{\alpha_1} \ldots \Delta_{\xi_n}^{\alpha_n}.
$
 For our further analysis we use the following result (see Corollary 4.5.7 of \cite{Ruz}):

\begin{lemma}\label{IClases}\ Let $0\leq \delta \leq 1,$ $0\leq \rho<1.$ Let $a:\mathbb{T}^n\times \mathbb{R}^n\rightarrow \mathbb{C}$ satisfying
\begin{equation}
\lvert  \partial_\xi ^\alpha \partial_x^\beta a (x, \xi)\rvert \leq C^{(1)}_{a\alpha \beta m}\langle \xi \rangle^{m- \rho \lvert \alpha \rvert + \delta \lvert \beta \rvert} ,
\end{equation}
for $|\alpha|\leq N_{1}$ and $|\beta|\leq N_{2}.$ Then the restriction $\tilde{a}=a|_{\mathbb{T}^n\times \mathbb{Z}^n}$ satisfies the estimate
\begin{equation}
\lvert  \Delta_\xi ^\alpha \partial_x^\beta \tilde{a} (x, \xi)\rvert \leq C_{a\alpha \beta m}C^{(1)}_{a\alpha \beta m}\langle \xi \rangle^{m- \rho \lvert \alpha \rvert + \delta \lvert \beta \rvert} ,
\end{equation}
for $|\alpha|\leq N_{1}$ and $|\beta|\leq N_{2}.$  The converse holds true, i.e, if a symbol $\tilde{a}(x,\xi)$ on $\mathbb{T}^n\times \mathbb{Z}^n$ satisfies  $(\rho,\delta)$-inequalities of the form
\begin{equation}
\lvert  \Delta_\xi ^\alpha \partial_x^\beta \tilde{a} (x, \xi)\rvert \leq C^{(2)}_{a\alpha \beta m}\langle \xi \rangle^{m- \rho \lvert \alpha \rvert + \delta \lvert \beta \rvert} ,
\end{equation}
then $\tilde{a}(x,\xi)$ is  the restriction of a symbol $a(x,\xi)$ on $\mathbb{T}^n\times \mathbb{R}^n$ satisfying estimates of the type
\begin{equation}
\lvert  \partial_\xi ^\alpha \partial_x^\beta a (x, \xi)\rvert \leq C_{a\alpha \beta m}C^{(2)}_{a\alpha \beta m}\langle \xi \rangle^{m- \rho \lvert \alpha \rvert + \delta \lvert \beta \rvert} .
\end{equation}
\end{lemma}
So, we prove the following result.

\begin{theorem}[Calder\'on-Vaillancourt, discrete version]
    Let us assume that $t_m$ is a pseudo-differential operator on $\mathbb{Z}^n.$ Then, under the condition \begin{equation}\label{calderon}
  |\partial_x^\beta\Delta_\xi^\alpha m(x,\xi)|\leq C_{\alpha,\beta}(1+|\xi|)^{(|\beta|-|\alpha|)\rho},\,\,0\leq \rho<1\,\,
\end{equation} the operator $ t_m $ extends to a bounded operator on $\ell^2(\mathbb{Z}^n).$
  \end{theorem} \begin{proof}
                  Let us define the function $a$ on $\mathbb{T}^n\times \mathbb{R}^n$ by the identity $m(n',\xi)=\overline{a(\xi,-n')}.$ Then we have the symbol inequalities
                  \begin{equation}\label{calderon''}
  |\partial_x^\beta\Delta_\xi^\alpha m(x,\xi)|\leq C_{\alpha,\beta}(1+|\xi|)^{(|\beta|-|\alpha|)\rho},\,\,0\leq \rho<1.
            \end{equation} By Lemma \ref{IClases}, there exists $\tilde{a}$ on $\mathbb{T}^n\times \mathbb{R}^n$ such that $a=\tilde{a}|_{\mathbb{T}^n\times \mathbb{Z}^n}$ satisfying \begin{equation}\label{calderon'}
  |\partial_x^\beta\partial_\xi^\alpha \tilde{a}(x,\xi)|\leq C_{\alpha,\beta}(1+|\xi|)^{(|\beta|-|\alpha|)\rho},\,\,0\leq \rho<1.
            \end{equation} By the continuous Calder\'on-Vaillancourt Theorem (Theorem \ref{CalVaillTheo}) the pseudo-differential operator $T$ on $\mathbb{R}^n$ with symbol $\tilde{a}$ extends to a bounded operator on $L^2(\mathbb{R}^n).$ So, by Theorem \ref{teorema principal1} we obtain the $\ell^2(\mathbb{Z}^n)$-boundedness of $t_m.$ So, we finish the proof.
                \end{proof}

We end this paper  with the following characterization which is a discrete version of the Gohberg Lemma proved by  Molahajloo, Ruzhansky and Dasgupta (see Molahajloo \cite{s1} and Dasgupta and Ruzhansky \cite{DRgohberg}).
\begin{theorem}[Gohberg Lemma, discrete version]
   A pseudo-differential operator $t_m$ on $\mathbb{Z}^n$ extends to a compact operator on $\ell^2(\mathbb{Z}^n)$ if and only if
      \begin{eqnarray}
      \limsup_{|n'|\rightarrow\infty}\sup_{\xi\in\mathbb{T}^n}|m(n',\xi)|=0.
      \end{eqnarray}
\end{theorem}

\begin{proof}
  By equation \ref{Ruzidentity}, we have $$
  t_m=\mathscr{F}^{-1}A^{*}\mathscr{F},\,\,\,m(n',\xi)=\overline{a(\xi,-n')},$$ where $A$ is the periodic operator associated to $a.$ So, $t_m$ is compact on $\ell^2(\mathbb{Z}^n)$ if and only if $A^*$ (and hence $A$) extends to a bounded operator on $L^2(\mathbb{T}^n).$ Now, $A$ is compact on $L^2(\mathbb{T}^n)$ (see \cite{DRgohberg} with $G=\mathbb{T}^n$) if and only if
  \begin{equation}\label{gohberg}
    \limsup_{|n'|\rightarrow\infty}\sup_{\xi\in\mathbb{T}^n}|a(\xi,n')|=0.
  \end{equation} So we obtain
  \begin{equation}\label{gohberg}
    \limsup_{|n'|\rightarrow\infty}\sup_{\xi\in\mathbb{T}^n}|m(n',\xi)|=0.
  \end{equation} Thus, we end the proof.
\end{proof}

\noindent \textbf{Acknowledgments:} I would like to thank the anonymous referee for his/her remarks which helped to improve the manuscript. This project was partially supported by Pontificia Universidad Javeriana, Mathematics Department, Bogot\'a-Colombia.
\bibliographystyle{amsplain}

\begin{thebibliography}{99}

\bibitem{CalderonZygmund}  Calder\'on, A. P., Zygmund, A. On the existence of certain singular integrals, Acta.
Math. 88 (1952), 85--139.

\bibitem{ca1} Calder\'on, A., Vaillancourt, R.: {On the boundedness of pseudo-differential operators.} J. Math. Soc. Japan \textbf{23}, (1971), 374--378.

\bibitem{ca2} Calder\'on, A., Vaillancourt, R.: {A class of bounded pseudo-differential operators} Proc. Nat. Acad. Sci. Usa.  \textbf{69},  (1972), 1185--1187.
 

\bibitem{car} Carneiro, E. and  Hughes, K. "On the endpoint regularity of discrete maximal operators." Math. Res. Lett. {19}(6), (2012), 1245--1262.


 \bibitem{carro} Carro, M., "Discretization of linear operators on $L^p(\mathbb{R}^n )$," Illinois J. Math. {42}(1) (1998), 1--18.

\bibitem{Cat14}  Catana V.,. $L^p$-boundedness of multilinear pseudo-differential operators on $\mathbb{Z}^n$ and $\mathbb{T}^n.$
Math. Model. Nat. Phenom., 9(5), (20014), 17--38.

\bibitem{DRgohberg} Dasgupta, A., Ruzhansky, M. The Gohberg lemma, compactness, and essential spectrum of operators on compact Lie groups. J. Anal. Math.  128,  (2016), 179--190.

\bibitem{DW} Delgado, J., Wong, M.W.: $L^p$-nuclear pseudo-differential operators on $\mathbb{Z}$ and $\mathbb{S}^1.,$  Proc. Amer. Math. Soc., 141(11), (2013), 3935--394.


\bibitem{Duo} Duoandikoetxea, J.: {Fourier Analysis}, Amer. Math. Soc.  (2001)

\bibitem{GBN} Ghaemi, M. B., Jamalpour Birgani, M.   Nabizadeh Morsalfard, E. A study on
pseudo-differential operators on $\mathbb{S}^1$ and $\mathbb{Z}$. J. Pseudo-Differ. Oper. Appl., 7(2), (2016), 237--
247.

\bibitem{GrafakosBook} L. Grafakos. Classical Fourier Analysis
Grad. Texts in Math., vol. 249, Springer-Verlag, New York (2008)

\bibitem{Hooley} Hooley, C. On Hypothesis $K^*$ in Waring’s problem, Sieve Methods, Exponential Sums, and their
Applications in Number Theory, London Math. Soc. Lecture Notes No. 237, Cambridge University Press,
(1997),  175--185.

\bibitem{Hormander1960} H\"ormander, L.  Estimates for translation invariant operators in $L^p$
spaces. Acta Math., 104, (1960),  93--140.

\bibitem{tn} Hughes, K. J., "Arithmetic analogues in harmonic analysis: Results
related to Waring's problem," PhD Thesis. Princeton University, (2012)

\bibitem{k} Kikuchi, N., Nakai, E., Tomita, N., Yabuta, K., and  Yoneda, T.   "Calder\'on-Zygmund operators on amalgam spaces and in the discrete case," J. of Math. Anal. and Appl. {335}(1), (2007), 198--212.


\bibitem{m}  Molahajloo, S. { Pseudo-differential Operators on $\mathbb{Z}$} in Pseudo-differential Operators: Complex analysis and partial differential equations, Operators Theory, Advances and Applications. {205} (2010), 213--221.
    
\bibitem{s1} Molahajloo, S.: {A characterization of compact pseudo-differential operators on $\mathbb{S}^1$ } Oper. Theory Adv. Appl. Birkh\"user/Springer Basel AG, Basel. \textbf{213}, 25-29 (2011)

\bibitem{Pie} Pierce, L. Discrete Analogues in Harmonic Analysis. P.h.D Thesis, Princeton University (2009).

\bibitem{Pie2}  Pierce, L.   On discrete fractional integral operators and mean values of Weyl sums.  Bull. London Math. Soc., 43 (2011) 597--612.

\bibitem{Ruzhansky} Botchway L., Kibiti G., Ruzhansky M., Difference equations and pseudo-differential operators on $\mathbb{Z}^n$, arXiv:1705.07564.

\bibitem{IW} Ionescu, A.D.,  Wainger, S.  $L^p$ boundedness of discrete singular Radon transforms, J. Amer. Math.
Soc. 19(2) (2005),  357--383.

\bibitem{Rab1} Rabinovich, V. Exponential estimates of solutions of pseudodifferential equations on the
lattice $(\mu\mathbb{Z})^n$: applications to the lattice Schr\"odinger and Dirac operators. J. Pseudo-Differ. Oper. Appl., 1(2), (2010), 233--253.

\bibitem{Rab2} Rabinovich, V. . Wiener algebra of operators on the lattice $(\mu\mathbb{Z})^n$ depending on the
small parameter $\mu>0$. Complex Var. Elliptic Equ., 58(6), (2013), 751--766.

\bibitem{Rab3}  Rabinovich, V. S.,  Roch, S. The essential spectrum of Schr\"odinger operators on lattices. J. Phys. A, 39(26), (2006), 8377--8394.

\bibitem{Rab4} Rabinovich, V. S.,  Roch, S. Essential spectra and exponential estimates of eigenfunctions
of lattice operators of quantum mechanics. J. Phys. A, 42(38), (2009), 385--207.

\bibitem{rod} Rodriguez, C. A.,  {$L^p-$estimates for pseudo-differential operators on $\mathbb{Z}^n$, } J. Pseudo-Differ. Oper. Appl.   {1}, (2011),  183--205.

\bibitem{riesz} Riesz, M., "Sur les fonctions conjuge\'ees," Math. Z. 27 (1928), 218--244.

\bibitem{Ruz} Ruzhansky, M., Turunen, V.: {Pseudo-differential Operators and Symmetries: Background Analysis and Advanced Topics } Birkha\"user-Verlag, Basel, (2010)

\bibitem{Stepanov} Stepanov, V. D. On convolution integral operators, Soviet Math. Dokl. 19 (1978), 1334--1337

\bibitem{st}  Stein, E., Wainger, S.,  Discrete analogues of singular Radon transforms, Bull. Amer. Math. Soc. (N.S.) 23 (1990), no. 2, 537-544.

\bibitem{st2}Stein, E., Wainger, S., Discrete analogues in harmonic analysis, I. $l^2$ estimates for singular Radon transforms," Amer. J. Math. 121(6) (1999),   1291--1336.

\bibitem{st3} Stein, E., Wainger, S. Discrete analogues in harmonic analysis, II. Fractional integration, J. Anal. Math. 80 (2000),  335--355.
    
\bibitem{SteWei} Stein, E., Weiss, G. Introduction to Fourier analysis on Euclidean spaces. Princeton University Press, Princeton, N.J., 1971



\end{thebibliography}

\end{document}